\documentclass[11pt]{article}
\usepackage{amsmath}
\usepackage{amsthm}
\usepackage{amsfonts}
\usepackage{setspace}
\usepackage{fullpage}
\usepackage{amssymb}
\usepackage{enumitem}

\usepackage{comment}

\usepackage{tikz}
\usepackage{float}
\usepackage{caption}

\newtheorem{lemma}{Lemma}
\newtheorem{theorem}{Theorem} 
\newtheorem{proposition}{Proposition} 
 
\newtheorem{definition}{Definition}
\newtheorem{problem}{Problem}
\newtheorem{claim}{Claim}

\newcommand{\EH}{{\rm EH}}

\date{}
\newcounter{num}

\newcommand{\calI}{\mathcal{I}}

\newcommand{\calS}{\mathcal{S}}
\newcommand{\Bin}{\mathrm{Bin}}

\title{Large cliques or co-cliques in hypergraphs with forbidden order-size pairs}

\author{Maria Axenovich\thanks{Karlsruhe Institute of Technology, Karlsruhe, Germany, 
	\texttt{maria.aksenovich@kit.edu}.	Research supported in part by the DFG grant FKZ AX 93/2-1.} \and Domagoj Brada\v{c}\thanks{Department of Mathematics, ETH, Z\"urich, Switzerland, \texttt{domagoj.bradac@math.ethz.ch}. Research supported by SNSF grant 200021\_196965.}
	\and Lior Gishboliner\thanks{Department of Mathematics, ETH, Z\"urich, Switzerland, \texttt{lior.gishboliner@math.ethz.ch}. Research supported by SNSF grant 200021\_196965.}
	\and Dhruv Mubayi\thanks{University of Illinois at Chicago, Chicago, USA, 	 \texttt{mubayi@uic.edu}. Research partially supported by NSF grants DMS-1763317, 1952767, 2153576, a Humboldt research award and a Simons fellowship.} \and Lea Weber\thanks{Karlsruhe Institute of Technology, Karlsruhe, Germany, \texttt{lea.weber@kit.edu}.}}

\begin{document}

\maketitle

\begin{abstract}The well-known Erd\H os-Hajnal conjecture states that for any  graph $F$, there exists $\epsilon>0$ such that  every $n$-vertex graph $G$ that contains no induced copy of $F$ has a homogeneous   set of size at least $n^{\epsilon}$. We consider a variant of the Erd\H{o}s-Hajnal problem for hypergraphs where we forbid a family of hypergraphs described by their orders and sizes. 
For graphs, we observe that if we forbid induced subgraphs on $m$ vertices and $f$ edges for any positive $m$ and $0\leq f \leq \binom{m}{2}$, then we obtain large homogeneous sets. For triple systems, in the first nontrivial case $m=4$, for every $S \subseteq \{0,1,2,3,4\}$, we give bounds on the minimum size of a homogeneous set in a triple system where the number of edges spanned by every four vertices is not in $S$. 
In most cases the bounds are essentially tight. We also determine, for all $S$, whether the growth rate is polynomial or polylogarithmic. Some open problems remain. 
\end{abstract}

\section{Introduction}
For an integer $r\geq 2$, an $r$-{\it graph} or $r$-uniform hypergraph is a pair $H=(V, E)$, where $V=V(H)$ is the set of vertices and  $E=E(H) \subseteq \binom{V}{r}$ is the set  of edges.  A $2$-graph is simply a graph.  A {\it homogeneous set} is a set of vertices that is either a clique or a coclique (independent set). For an $r$-graph $H$, let $h(H)$ be the size of a largest homogeneous set.  Given $r$-graphs $F, H$, say that $H$ is $F$-{\it free}  if $H$ contains no isomorphic copy of $F$ as an induced subgraph.
We say that an $r$-graph $F$ has the {\it Erd\H{o}s-Hajnal-property} or simply {\it {\rm EH}-property} if there is a constant $\epsilon=\epsilon_F>0$ such that every $n$-vertex $F$-free $r$-graph $H$ satisfies $h(H) \geq n^{\epsilon}$.  A conjecture of Erd\H{o}s and Hajnal~\cite{EH} states that  any $2$-graph has the EH-property.  The conjecture remains open, see for example  a survey by Chudnovsky~\cite{C}, as well as \cite{APS, BLT, FPS}, to name a few.   When $F$ is a fixed graph and $G$ is an $F$-free $n$-vertex graph, Erd\H{o}s and Hajnal proved that $h(G) \ge 2^{c\sqrt{\log n}}$. This was recently improved to $h(G) \ge 2^{c\sqrt{\log n\log\log n}}$ by   Buci\'{c}, Nguyen, Scott, and Seymour~\cite{BNSS}.

The Erd\H{o}s-Hajnal conjecture fails for $r$-graphs, $r\geq 3$, already when $F$  is a clique of size $r+1$.  Indeed,  well-known results on off-diagonal hypergraph Ramsey numbers show that there are $n$-vertex $r$-graphs that do not have a clique on $r+1$ vertices and do not have cocliques on $f_r(n)$ vertices, where $f_r$ is an iterated logarithmic function (see~\cite{MS} for the best known results). 
 Moreover, the following   result (Claim 1.3. in \cite{GT})   tells us exactly  which $r$-graphs, $r\geq 3$,  have the EH-property. Here $D_2$ is the unique $3$-graph on $4$ vertices with exactly $2$ edges.

\begin{theorem}[Gishboliner and Tomon~\cite{GT}]\label{GT}
Let $r\geq 3$. If $F$ is an $r$-graph on at least $r+1$ vertices and $F\neq D_2$, then there is an $F$-free  $r$-graph  $H$ on $n$ vertices such that $h(H) =(\log n)^{O(1)}$.
\end{theorem}

It is natural to consider the EH-property for families of $r$-graphs instead of a single $r$-graph. 
In this paper, we consider families determined by a given set of orders and sizes. Several special cases of this have been extensively studied over the years (see, e.g.~\cite{EH1}).
 For  $0\leq f \leq \binom{m}{r}$,   we call an $r$-graph $F$  on $m$ vertices and $f$ edges an $(m,f)$-\emph{graph} and we call the pair $(m,f)$ the {\it order-size pair} for $F$.
 Say that $H$ is $(m,f)$-free if it contains no induced copy of an $(m,f)$-graph.
 If $Q=\{(m_1, f_1), \ldots, (m_t, f_t)\}$, say that $H$ is $Q$-free if $H$ is $(m_i,f_i)$-free for all $i=1, \ldots, t$.
 
 \begin{definition}
 	Given $r \ge 2$ and 
 		 $Q=\{(m, f_1), \ldots, (m, f_t)\}$, let $h(n,Q)=h_r(n,Q)$ be the minimum of $h(H)$, taken over all  $n$-vertex $Q$-free $r$-graphs $H$.  Say that  $Q$  has the \EH-{\it property} if there exists $\epsilon=\epsilon_Q>0$ such that $h(n, Q) >n^{\epsilon}$. 
 \end{definition}
 
 For example $h_3(n, \{(4,0), (4,2)\}) = k$ means that any $n$-vertex $3$-graph in which  any $4$ vertices induce $1$, $3$, or $4$ edges has a homogenous set of size $k$, and there is an $r$-graph $H$ as above with $h(H) = k$.  We may omit the subscript $r$ in the notation $h_r(n, Q)$ if it is obvious from context. When $Q=\{(m,f)\}$ we use the simpler notation $h(n,m,f)$ instead of $h(n, \{(m,f)\})$.  Let us make two simple observations:  \begin{equation} \label{subset}
 	h_r(n, Q) \le h_r(n, Q') \qquad \text{ if} \qquad  Q\subseteq Q',
 	\end{equation}

 \begin{equation} \label{complement} h_r(n, Q) = h_r(n, \overline Q) \qquad \text{ where } \qquad 
 	\overline Q = \left\{\left(m, {m \choose r}-f\right): (m,f) \in Q\right\}.
 	\end{equation}
 Our first result concerns 2-graphs, where we show that forbidding a single order-size pair already guarantees large homogeneous sets.
 
 \begin{proposition}\label{graph}
 For any integers  $m,f$ with  $m\geq 2$ and $0\leq f\leq \binom{m}{2}$ there exists $c>0$ such that 
 %every $n$-vertex graph $G$ with no induced $(m,f)$-subgraph satisfies $h(G)> c\, n^{1/(m-1)}$.
 $h_2(n, m, f) > c \, n^{1/(m-1)}$.
  \end{proposition}

% Proposition~\ref{graph} shows that $h(n, m, f) > c \, n^{1/(m-1)}$ and it seems a challenging problem to give good upper bounds. For example, determining $h(n,m,{m \choose 2})$ is equivalent to determining off-diagonal Ramsey numbers.  

It seems a challenging problem to give good upper bounds on $h_2(n, m, f)$. For example, determining $h_2(n,m,{m \choose 2})$ is equivalent to determining off-diagonal Ramsey numbers.  
 
 Our remaining results are in the hypergraph case $r=3$ and $m=4$. 
  We shall be considering sets $Q$ of pairs $(4,i)$ for $i\in \{0, 1, 2, 3, 4\}$. 
 We do not need to consider sets $Q$ that contain both $(4,0)$ and $(4,4)$ because Ramsey's theorem guarantees that for sufficiently large $n$ we cannot avoid both of them. 
 Using complementation~(\ref{complement}), this leaves us with the following sets: 
 \begin{itemize}
 	\item{} $\{(4,0)\}$, ~$\{(4,1)\}$, ~$\{(4,2)\}$; \\
 	\item{}  $ \{ (4,0), (4,1) \}$,~  $ \{ (4,0), (4,2) \}$,   ~$ \{ (4,0), (4,3) \}$,  ~$ \{ (4,1), (4,2) \}$, ~$ \{ (4,1), (4,3) \}$; ~~ and \\
 	\item{} $\{ (4,0), (4,1), (4,2) \}$,  ~$\{ (4,0), (4,1), (4,3) \}$, ~$\{ (4,0), (4,2), (4,3) \}$, ~$\{ (4,1), (4,2), (4,3) \}$.
 \end{itemize} 
We address $h(n, Q)$ for each of these choices of $Q$.

We quickly obtain bounds for the first case using results in Ramsey theory (note again that $h(n,4,f)=h(n,4,4-f)$). Recall that the Ramsey number  $R_k(s,t)$ is the minimum $n$ such that every red/blue edge-coloring of the complete $n$-vertex $k$-graph yields either a monochromatic red $s$-clique or a monochromatic blue $t$-clique. It is known~\cite{CFS} that $2^ {c t\log t} \leq  R_3 (4, t)  \leq   2^{c' t^2 \log t}$. This yields positive constants $c$ and $c'$, such that
	\begin{equation} \label{40lower}c' \left(\frac{\log n}{\log\log n}\right)^{1/2} < h_3(n, 4,0)  <   c\frac{\log n}{\log\log n}.
		\end{equation}

 A more recent result of Fox and He~\cite{FH} constructs $n$-vertex $3$-graphs with every four vertices spanning at most two edges and independence number at most $c\log n/\log\log n$. Together with (\ref{subset}) this yields positive a constant  $c$, such that
 	\begin{equation} \label{41upper} h_3(n, 4,1)  \le
 		h_3(n, \{(4,0), (4,1)\})
 		<  c \frac{\log n}{\log\log n}.
 \end{equation}
 
 For the remaining cases when $|Q|=1$ we obtain  bounds using recent results by Fox and He~\cite{FH} and  by Gishboliner and Tomon~\cite{GT}. 
 % Recall that $f(n) = (1+o(1))g(n)$ means that there is a function $e(n)$ such that $\lim_{n \rightarrow \infty} e(n) \rightarrow 0$ and $f(n) = g(n) + e(n)g(n)$. 

\begin{proposition} \label{thm:41} There are  positive constants $c_1, c_2$  such that 
	  \begin{equation} \label{eq:41}   h_3(n, 4,1 ) > c_1 \, \left(\frac{\log n}{\log\log n}\right)^{1/2}    \end{equation}
and 	  \begin{equation} \label{eq:42} n^{c_1}< h_3(n,4,2) < c_2n^{1/3}\log^{4/3} n. \end{equation}
	  	\end{proposition} 
 It is unclear if either bound for $h(n,4,1)$ above represents the correct order of magnitude, but the lower bound certainly seems far off. 
 
\begin{problem} Improve the exponent $1/2$ in the lower bound on $h_3(n,4,1)$.
%Theorem~\ref{thm:41}.
	\end{problem}

 Our next results address  the case when $|Q|=2$. 
 Here we have essentially tight bounds in all cases. First, for $Q = \{(4,0),(4,1)\}$, it is known that 
 $$c \, \frac{\log n}{\log\log n}< h_3(n,\{(4,0), (4,1)\}) < c' \frac{\log n}{\log\log n}$$
 for some constants $c,c' > 0$. 
 The lower bound follows (after applying (\ref{complement})) from an old result of Erd\H os and Hajnal~\cite{EH1}. This is the first instance of a (different) conjecture of Erd\H os and Hajnal~\cite{EH1} about the growth rate of generalized hypergraph Ramsey numbers that correspond to our setting of  $h(n, Q)$, 
 where $Q = \{ (m, f), (m, f+1), \ldots, (m, \binom{m}{r}) \}$. Recent results of Mubayi and Razborov~\cite{MR} on this problem determine, for each $m>r \ge 4$, the minimum $f$ such that $h_r(n, Q) < c \log ^an$ for  some $a$ and  $Q=\{(m,f), \ldots, (m, \binom{m}{r})\}$. When $r=3$, the  minimum $f$ was determined by Conlon, Fox, and  Sudakov \cite{CFS} for $m$ being a power of $3$ and for growing $m$, as well as some other values.

 The following theorem give essentially tight bounds on $h_3(n,Q)$ for each of the remaining $Q$ with $|Q| = 2$.
\begin{theorem}\label{thm:|Q|=2}
There are positive constants $c_1,c_2$ such that:
\begin{enumerate}[label=(\alph*)]
\item   $c_1\sqrt{n} \leq h_3(n,\{(4,0),(4,2)\}) \leq c_2\sqrt{n\log n}.$
  \item  $c_1 n^{1/3}\log^{1/3}n \leq h_3(n,\{(4,1),(4,2)\}) \leq c_2n^{1/3} \log^{4/3} n$ 
\item     $c_1 n \leq h_3(n,\{(4,0),(4,3)\})$
    \item $c_1\log n \leq h_3(n,\{(4,1),(4,3)\}) \leq c_2 \log n$. 
\end{enumerate}
\end{theorem}
 
 % \begin{theorem}\label{thm:24}
 % 	$\Omega(\sqrt{n}) = h_3(n,\{(4,0),(4,2)\}) = O(\sqrt{n\log n})$.
 % \end{theorem}
 
 % \begin{theorem}\label{thm:23}
 % 	$\Omega(n^{1/3}\log^{1/3}n) \leq h_3(n,\{(4,1),(4,2)\}) = O(n^{1/3} \log^{4/3} n)$.
 % \end{theorem} 

% \begin{theorem}\label{thm:03}
% 	$h_3(n,\{(4,0),(4,3)\}) = \Theta(n)$.
% \end{theorem}

% \begin{theorem}\label{thm:13}
% 	$h_3(n,\{(4,1),(4,3)\}) = \Theta(\log n)$.
% \end{theorem}

It would be interesting to close the (polylogarithmic) gap between the upper and lower bounds in the first two items of Theorem~\ref{thm:|Q|=2}. See also the remark at the end of Section \ref{sec:24}. 

Regarding $h_3(n,\{(4,0),(4,3)\})$, we have the easy upper bound $n/3 +1 $ by taking a partition of the vertex set into three cliques $A_1, A_2, A_3$ of almost equal sizes and adding all triples with exactly once vertex in $A_i$ and two vertices in $A_{i+1}$ for $i=1,2,3$, where indices are taken modulo 3.
 
 It is worth mentioning that $\{(4,1),(4,3)\}$-free $3$-graphs are also known as {\em two-graphs}, and have been thoroughly studied in algebraic combinatorics due to their connection to sets of equiangular lines, see e.g. \cite[Chapter 11]{GR}. Every two-graph $H$ arises from some graph $G$ by taking $x,y,z \in V(G)$ to be an edge of $H$ if and only if $\{x,y,z\}$ induces an odd number of edges in $G$. We will use this connection in the proof of Theorem~\ref{thm:|Q|=2}(d).

 Finally, we consider the case when $|Q|=3$. If $Q=\{(4,0), (4,1), (4,2)\}$, then a $\overline Q$-free 3-graph is a partial Steiner system (STS), and it is well known~\cite{EHR, PR, DPR} that the minimum independence number of an $n$-vertex partial STS has order of magnitude $\sqrt{n \log n}$. Thus
 $h_3(n, Q)$ has order of magnitude  $\sqrt{n \log n}$.  If $Q=\{(4,1),(4,2), (4,3)\}$, and $n \ge 4 $, then it is a simple exercise to show that any $Q$-free 4-graph on at least four vertices is a clique or coclique and therefore $h_3(n, Q)=n$ for $n \ge 4$.
 The two remaining cases are covered below.
 \begin{theorem} \label{|Q|=3} Let $n \ge 4$. Then $h_3(n, \{(4,0),(4,2), (4,3)\}) =n-1$ and
 	
 	\begin{equation}\label{013} h_3(n, \{(4,0), (4,1), (4,3)\}) =\begin{cases}
 			 \frac{n}{2}  &\text {if  $n \equiv 0$ (mod 6)} \\
 \lceil \frac{n+1}{2}\rceil & \text {if  	$n \not\equiv 0$ (mod 6).}
 \end{cases}	  \notag 
 	    \end{equation}
 	 \end{theorem}

\medskip

{\bf Paper organization:}
The rest of this paper is organized as follows. In Section 2 we prove Proposition~\ref{graph}. In Section \ref{sec:short} we prove Proposition~\ref{thm:41} and Theorem~\ref{thm:|Q|=2}(c). In Sections \ref{sec:24}, \ref{sec:23} and \ref{sec:13} we prove parts (a), (b) and (d) of Theorem~\ref{thm:|Q|=2}, respectively. 
Finally, Theorem~\ref{|Q|=3} is proved in Section \ref{sec:|Q|=3}. 
The last section contains some concluding remarks.

{\bf Notation:}
Throughout the paper, for a hypergraph $H$, let $\omega(H)$ and $\alpha(H)$ denote the size of a largest clique and independent set in $H$, respectively. 
Recall that $h(H) = \max\{\omega(H),\alpha(H)\}$. 
% For an $r$-graph $H$ and one of its vertices $v$, we define the {\it link graph} of $v$ to be the $(r-1)$-graph $L(v)$  whose vertex set is $V(H)\setminus \{v\}$ and edge set is $\{e \subseteq V(H)\setminus \{v\}: e \cup \{v\} \in E(H)\}$.
For a $3$-graph $H$ and one of its vertices $v$, we define the {\it link graph} of $v$ to be the graph $L(v)$ whose vertex set is $V(H)\setminus \{v\}$ and whose edge set is $\{e \subseteq V(H)\setminus \{v\}: e \cup \{v\} \in E(H)\}$.
Moreover, for $S \subseteq V(H) \setminus \{v\}$, we use $L_S(v)$ to denote the subgraph of $L(v)$ induced by $S$. A clique on $s$ vertices is denoted $K_s$.
When denoting edges in $3$-graphs, we shall often omit parentheses and commas; for example, instead of writing $\{x,y,z\}$, we shall simply write $xyz$.
A {\em star} is a hypergraph consisting of a set $S$ and a vertex $v \notin S$ with edge-set $\{vxy : x,y \in S, x \neq y\}$. We will denote this star by $(v,S)$.
As usual, we write $f(n) = O(g(n))$ if there is a constant $C > 0$ such that $f(n) \leq C g(n)$ for all $n$, and we write $f(n) = \Omega(g(n))$ to mean that $g(n) = O(f(n))$. 

\section{Graphs: Proof of Proposition~\ref{graph}}\label{graphs}

%In this section we prove Proposition~\ref{graph}. 
%For a graph $G$, let $\omega(G)$ and $\alpha(G)$ denote the size of a largest clique and coclique, respectively.
% \medskip

\noindent
{\bf Proof of Proposition~\ref{graph}.}
We shall use induction on $m$ with basis $m=2$. In this case $f\in \{0,1\}$.  Note that $h(n, 2, 0) =h(n, 2, 1) =n = n^1 = n^{1/(m-1)}$, 
since forbidden graphs are either a non-edge or an edge. 
Consider an $(m, f)$-free graph $G$ on $n$ vertices, $m\geq 3$, and assume that the statement of the proposition holds for smaller values of $m$.
We can also assume that $G$ is not a complete graph, an empty graph, a cycle, or the complement of a cycle, since we are done in these cases.
Consider $\Delta$ and $\overline{\Delta}$, the maximum degree of $G$ and of the complement $\overline{G}$ of $G$, respectively.
Using Brooks' theorem,  the chromatic number of $G$ and  of $\overline{G}$ is at most $\Delta$ and $\overline{\Delta}$, respectively.
Thus, $\alpha(G)\geq n/\Delta$ and  $\omega(G)\geq n/\overline{\Delta}$. Therefore, we can assume that $\Delta \geq n^{(m-2)/(m-1)}$ and $\overline{\Delta}\geq n^{(m-2)/(m-1)}$, otherwise we are done. Thus, there is a vertex with at least $n^{(m-2)/(m-1)}$ edges incident to it and there is a vertex with at least $n^{(m-2)/(m-1)}$ non-edges incident to it.

Assume first that $f\leq m-1$. Consider a vertex $v$ with at least $n^{(m-2)/(m-1)}$ non-edges incident to it, i.e., with a set $X$ of vertices each non-adjacent to $v$, $|X|\geq n^{(m-2)/(m-1)}$. Since $G$ is $(m,f)$-free, $G[X]$ is $(m-1, f)$-free. Thus, by induction $h(G) \geq h(G[X]) \geq |X|^{1/(m-2)}\geq 
n^{1/(m-1)}.$

Now assume that $f\geq m$.  Consider a vertex $v$ with at least $n^{(m-2)/(m-1)}$ edges incident to it, i.e., with a set $X$ of vertices each adjacent to $v$, $|X|\geq n^{(m-2)/(m-1)}$. Since $G$ is $(m,f)$-free, $G[X]$ is $(m-1, f-(m-1))$-free.  Thus, by induction 
$h(G) \geq h(G[X]) \geq |X|^{1/(m-2)}\geq 
n^{1/(m-1)}.$ \qed

\section{Two Short Proofs}\label{sec:short}
\subsection{Proof of Proposition~\ref{thm:41}}
%\begin{proof}[Proof of Proposition~\ref{thm:41}]

To prove the lower bound on $h(n, 4, 1)$, we shall consider the complementary setting and an arbitrary $n$-vertex $(4,3)$-free $3$-graph $H$. 
We need the following theorem of Fox and He \cite{FH}.
\begin{theorem}[Fox and He \cite {FH}, Thm. 1.4]\label{thm:FH}
	For all $t, s\geq 3$, any $3$-graph on more than $(2t)^{st}$ vertices contains either a coclique on $t$ vertices or a star $(v,S)$ with $|S| = s$. 
\end{theorem}
We shall apply Theorem \ref{thm:FH} with the largest possible $t=s$ such that $ (2t)^{st}<n$. 
In this case $t=s \geq c(\log n / \log \log n)^{1/2}$. If  $H$  has a coclique of size $t$, then $h(H) \geq t$ and we are done. Otherwise  $H$ contains a star $(v,S)$ with $|S| = s$. Note that $S$ induces a clique in $H$, because otherwise $v$ and three vertices of $S$ not inducing an edge give a $(4,3)$-subgraph.  Thus, $h(H) \geq s-1$. In each case $h(H) \geq c(\log n / \log \log n)^{1/2}$.

Now, we consider the case  $Q=\{(4,2)\}$.  The lower bound on $h(n, 4, 2)$ follows from a  result of Gishboliner and Tomon~\cite{GT}. The upper bound follows from Item 2 of Theorem~\ref{thm:|Q|=2}, as $h(n,4,2) \leq h(n,\{(4,1),(4,2)\})$. This completes the proof of Proposition~\ref{thm:41}.

%\end{proof}

\subsection{Proof of Theorem~\ref{thm:|Q|=2}(c)}
Let $H$ be a $\{(4,0),(4,3)\}$-free hypergraph. 
We may assume that $e(H) = \Omega(n^3)$, else $H$ is not $(4,0)$-free. 
(Indeed, if $e(H) = o(n^3)$, then the probability that a random set of $4$ vertices contains an edge is $o(1)$, so $H$ contains a $(4,0)$-subgraph.)
Fix $v \in V(H)$ with $e(L(v)) = \Omega(n^2)$. Note that $L(v)$ is induced $C_4$-free. Indeed, if $C$ is an induced $C_4$ in $L(v)$, then for each $A \subseteq V(C)$, $|A| = 3$, it holds that $A \notin E(H)$, because else $A \cup \{v\}$ spans exactly $3$ edges. This means that $V(C)$ spans $0$ edges, a contradiction. By a result of Gy\'arf\'as, Hubenko and Solymosi \cite{GHS}, an $n$-vertex graph with $\Omega(n^2)$ edges and no induced $C_4$ contains a clique of size $\Omega(n)$.
So $L(v)$ contains a clique $X$ of size $\Omega(n)$. For each $A \subseteq X$, $|A| = 3$, we have $A \in E(H)$ because else $A \cup \{v\}$ spans exactly $3$ edges. So $X$ is a clique in $H$, implying $\omega(H) = \Omega(n)$. This proves Theorem~\ref{thm:|Q|=2}(c).

\section{$(4,0),(4,2)$: Proof of Theorem~\ref{thm:|Q|=2}(a)}\label{sec:24}
It will be convenient to consider $Q = \{(4,2),(4,4)\}$ (which is equivalent to $\{(4,0),(4,2)\}$ via complementation).
To lower-bound $h_3(n,\{(4,2),(4,4)\})$, we prove the following characterization of $\{(4,2),(4,4)\}$-free $3$-graphs.
A {\em tight component} is a maximal (with respect to inclusion) set of edges $C$ such that for any distinct  $e_1,e_2 \in C$, there is a tight walk from $e_1$ to $e_2$, i.e. a sequence of edges $e_1 = f_1,\dots,f_k = e_2$ with $|f_i \cap f_{i+1}| = 2$.  We call a tight component a {\it star} if it is an edge-set of a star.
\begin{theorem}\label{thm:charac}
	A $3$-graph $H$ is $\{(4,2),(4,4)\}$-free if and only if every tight component is a star. 
\end{theorem}
\begin{proof}
	Suppose first that every tight component of $H$ is a star.
	If $H$ contains $4$ vertices spanning exactly $2$ or $4$ edges, then the edges on these vertices are in the same tight component, but a star does not contain $4$ vertices spanning exactly $2$ or $4$ edges, a contradiction. So $H$ is  $\{(4,2),(4,4)\}$-free.
	
	We now prove the other direction. Let $H$ be a $\{(4,2),(4,4)\}$-free $3$-graph. 
	Observe that for every star $(v,S)$ in $H$, the set $S$ is independent, because otherwise $H$ would not be $(4,4)$-free. 
	\begin{claim}\label{claim:|S|>=3}
		Let $(v,S)$ be a star with $|S| \geq 3$. There is no edge in $H$ of the form $uxy$ with $u \notin \{v\} \cup S$ and $x,y \in S$.
	\end{claim}
	\begin{proof}
		Suppose otherwise. The vertices $\{v,u,x,y\}$ must span exactly $3$ edges, because $vxy,uxy \in E(H)$ but $\{v,u,x,y\}$ cannot span $2$ or $4$ edges. Without loss of generality, suppose that $vux \in E(H)$, $vuy \notin E(H)$. Let $z \in S \setminus \{x,y\}$. Suppose first that $vuz \in E(H)$. Then $uyz \in E(H)$ because otherwise $\{v,u,y,z\}$ spans $2$ edges. This implies that $uxz \in E(H)$, because else $\{u,x,y,z\}$ spans $2$ edges. Now $\{v,u,x,z\}$ spans $4$ edges, contradiction. 
		Similarly, suppose that $vuz \notin E(H)$. Then $uyz \notin E(H)$ because else $\{v,u,y,z\}$ spans $2$ edges. This implies that $uxz \notin E(H)$, because else $\{u,x,y,z\}$ spans $2$ edges. Now, $\{v,u,x,z\}$ spans $2$ edges, contradiction.
	\end{proof}
	
	Now we complete the proof of the theorem. 
	Let $C$ be a tight component of $H$, and let us show that $C$ is a star. If $|C|=1$ (i.e. $C$ contains only one edge) then this is immediate, so suppose that $C$ contains at least $2$ edges. Let $e,f \in C$ with $|e \cap f| = 2$. Write $e = uvx, f = uvy$. Then exactly one of the triples $vxy, uxy$ is an edge, say $vxy \in E(H)$. So $C$ contains the edges of the star $(v,\{u,x,y\})$. 
	Let $S$ be a maximal subset of $V(H) \setminus \{v\}$ such that $C$ contains the edges of the star $(v,S)$, so $|S| \geq 3$. We claim that $C$ contains no other edges. Suppose otherwise. 
    Recall that $S$ induces no edges.
    So there must be an edge $e \in C$ which contains one vertex $w$ outside $\{v\} \cup S$ and two vertices $s,t$ in $\{v\} \cup S$. By Claim \ref{claim:|S|>=3}, it is impossible that $s,t \in S$. So suppose that $s = v, t \in S$. Fix an arbitrary $z \in S \setminus \{t\}$. We have $vzt \in E(H)$. Also, $vwt \in E(H)$ (because $s = v$). By Claim \ref{claim:|S|>=3}, $wzt \notin E(H)$, which implies that $vwz \in E(H)$ as otherwise $\{ v, w, t, z\}$ spans exactly two edges. As this holds for every $z \in S$, we get that $(v,S \cup \{w\})$ is a star contained in $C$, contradicting the maximality of $S$.  
\end{proof}
%In what follows, for a tight component (a star) $C = (v,S)$ of an $\{(4,2),(4,4)\}$-free $3$-graph, we denote $V(C) := \{v\} \cup S$. 
In what follows, for a tight component $C$ that is a star, we denote by $V(C)$ the vertex set of the respective graph and $e(C)=|C|$, the number of edges in $C$.
\begin{lemma}\label{lem:disjoint tight components}
	Let $C_1,C_2$ be distinct tight components of a $\{(4,2),(4,4)\}$-free $3$-graph. Then
	$|V(C_1) \cap V(C_2)| \leq 1$. 
\end{lemma}
\begin{proof}
	Suppose by contradiction that there are distinct $x,y \in V(C_1) \cap V(C_2)$. Note that in a star, every pair of vertices is contained in some edge of the star. Let $e_i$ be an edge of $C_i$ containing $x,y$, $i=1,2$. Then there is a tight walk between every edge of $C_1$ and every edge of $C_2$ by using the connection $e_1,e_2$. It follows that $C_1,C_2$ are in the same tight component, a contradiction.  
\end{proof}

Next, we prove a tight bound for the number of edges in a $\{(4,2),(4,4)\}$-free $3$-graph. The extremal case is when $H$ is a star.

\begin{proposition}\label{thm:edge count}
	For a $\{(4,2),(4,4)\}$-free $n$-vertex $3$-graph $H$, it holds that $e(H) \leq \binom{n-1}{2}$.
\end{proposition}
\begin{proof}
	Let $C_1,\dots,C_m$ be the tight connected components of $H$. Each edge is contained in a unique $C_i$, and $e(C_i) = \binom{|V(C_i)|-1}{2}$ because $C_i$ is a star. 
	Therefore, $e(H) = \sum_{i=1}^m \binom{|V(C_i)|-1}{2}$.
	Also, $\sum_{i=1}^m\binom{|V(C_i)|}{2} \leq \binom{n}{2},$
	% \begin{equation}\label{eq:constraint}
	% \sum_{i=1}^m\binom{|V(C_i)|}{2} \leq \binom{n}{2},
	% \end{equation}
	because each pair of vertices is contained in at most one $V(C_i)$, by Lemma \ref{lem:disjoint tight components}. 
	Let $f$ be the function $f(x) = x - \frac{1}{2}\sqrt{8x+1}+\frac{1}{2}$, so that $f(\binom{k}{2}) = \binom{k-1}{2}$. 
	Put $x_i = \binom{|V(C_i)|}{2}$, so that $f(x_i) = \binom{|V(C_i)|-1}{2}$. We have $\sum_{i=1}^m x_i \leq \binom{n}{2}$. The function $f$ is convex on $[0,\infty)$, so $\sum_{i=1}^m f(x_i)$ is maximized when exactly one of the $x_i$'s, say $x_1$, is non-zero. As $x_1 \leq \binom{n}{2}$, we have 
	$e(H) = \sum_{i=1}^m f(x_i) \leq f(\binom{n}{2}) = \binom{n-1}{2}$.  
\end{proof}

\begin{proof}[Proof of Theorem~\ref{thm:|Q|=2}(a)]
	The upper bound in Theorem~\ref{thm:|Q|=2}(a) follows from the fact that every linear $3$-graph is $\{(4,2),(4,4)\}$-free (this follows e.g. from Theorem~\ref{thm:charac}, because every tight component of a linear hypergraph has size $1$), and the well-known result that there exist linear $3$-graphs with independence number $O(\sqrt{n \log n})$ (which is tight), see \cite{EHR, PR, DPR}. 
	
	The lower bound in Theorem~\ref{thm:|Q|=2}(a) follows from Proposition \ref{thm:edge count} and the known fact that every $n$-vertex $3$-graph $H$ has an independent set of size $c n^{3/2}/\sqrt{e(H)}$. (To see this, take a random subset $X \subseteq V(H)$ by keeping each vertex with probability $p = c\sqrt{n/e(H)}$, and delete one vertex from each edge inside $X$.)
 %see \cite[Chapter 3, exercise 3]{AS}.
\end{proof}

It would be interesting to determine whether $h(n,\{(4,2),(4,4)\}) = \Theta(\sqrt{n\log n})$. As mentioned above, it is known that every linear $3$-graph has an independent set of size $\Omega(\sqrt{n\log n})$ and this is tight. Another construction of a $\{(4,2),(4,4)\}$-free hypergraph is to take a projective plane and put a star on each line (so that each star has roughly $\sqrt{n}$ vertices). It would be interesting to estimate the smallest possible independence number of such a hypergraph.

\section{$(4,1),(4,2)$: Proof of Theorem~\ref{thm:|Q|=2}(b)}\label{sec:23}
Here we consider $Q = \{(4,1),(4,2)\}$. By complementation, we may equivalently consider $Q = \{(4,2),(4,3)\}$. 
We begin with the lower bound in Theorem~\ref{thm:|Q|=2}(b). 
Here we need the following result of Kostochka, Mubayi, and Verstra\"ete~\cite{KMV} on independent sets in sparse hypergraphs. 
\begin{theorem} [Kostochka, Mubayi, and Verstra\"ete~\cite{KMV}] \label{thm:KMV}
	Suppose that  $H$ is an $n$-vertex 3-graph  in which every pair of vertices lies in at most $d$ edges, where $0<d<n/(\log n)^{27}$. Then $H$ has an independent set of size at least
	$c \sqrt{(n/d)\log (n/d)}$ where $c$ is an absolute constant. 
\end{theorem}

\begin{proof}[Proof of the lower bound in Theorem~\ref{thm:|Q|=2}(b)]
Let $H$ be an $n$-vertex $\{(4,2), (4,3)\}$-free $3$-graph, where $n$ is sufficiently large. Let $u,v$ be a pair of vertices in $H$ whose common neighborhood $S$ has maximum size $d>0$. Given vertices $x,y\in S$, the edges $xyu$ and $xyv$ are both in $H$, else $\{u,v,x,y\}$ induces a $(4,2)$- or $(4,3)$-graph. Next, any three vertices $x,y,z \in S$
	must form an edge of $H$, otherwise $\{u,x,y,z\}$ induces a $(4,3)$-graph. Therefore $S$ induces a clique in $H$ of size $d$. If $d>n^{0.4}$, say, then we are done as $h(H)\ge d$. Recalling that $n$ is large enough, we may assume that $d \le n^{0.4} < n/(\log n)^{27}$. Now Theorem~\ref{thm:KMV} yields a coclique in $H$ of size at least $c \sqrt{(n/d)\log n}$ for some positive constant $c$. Consequently, there is a constant $c'$ such that 
	$$h(H) \ge  \max \, \{d, \, c \sqrt{(n/d)\log n} \} > c' \,(n \log n)^{1/3}.$$
	Replacing $c'$ by a possibly smaller constant $c_1$ yields the result for all $n>4$.
\end{proof}
In the rest of this section, we prove the upper bound in Theorem~\ref{thm:|Q|=2}(b).
We begin with the following two lemmas, giving a structural characterization of $\{(4,2),(4,3)\}$-free $3$-graphs and rephrasing the problem of estimating $h_3(n,\{(4,2),(4,3)\})$ in terms of a certain extremal problem for (non-uniform) linear hypergraphs. 
\begin{lemma}
	Let $H$ be a $\{(4,2),(4,3)\}$-free $3$-graph. Then every two maximal cliques in $H$ intersect in at most one vertex.
\end{lemma}
\begin{proof}
	Let $X,Y$ be maximal cliques and suppose that $|X \cap Y| \geq 2$. Fix $u,v \in X \cap Y$ and $y \in Y \setminus X$. 
    Note that $uvy \in E(H)$.
    For every $x \in X \setminus \{u,v\}$, we have $uvx \in E(H)$, so we must have $uxy,vxy \in E(H)$, because else $\{u,v,x,y\}$ spans $2$ or $3$ edges. 
    Next, for every $x_1,x_2 \in X \setminus \{u\}$, we have 
    $ux_1y, ux_2y \in E(H)$, so we must also have $x_1x_2y \in E(H)$. It follows that $X \cup \{y\}$ is a clique, in contradiction to the maximality of $X$. 
 % We show that $X \cup Y$ is a clique, which contradicts the maximality of $X,Y$. First we show that all triples of the form $xzy$ with $x \in X, z \in X \cap Y, y \in Y$ are edges of $H$. We may assume that $x \not\in Y, y \not\in X,$ as otherwise the conclusion follows trivially since $X, Y$ are cliques. Let $w \in X \cap Y$ be different from $z.$ Since $X, Y$ are cliques, we have $xzw, yzw \in E(H),$ so the set $\{x, y, z, w\}$ spans at least two edges, which implies that $xzy, xwy \in E(H)$ as well.
	%
	% Now, consider a triple of the form $x_1x_2y$ with $x_1, x_2 \in X \setminus Y$ and $y \in Y \setminus X.$ Fix $z \in (X \cap Y).$ Since $x_1x_2z, x_1yz, x_2yz$ are edges, and $\{x_1,x_2,y,z\}$ cannot span $3$ edges, it must be that $x_1x_2y$ is also an edge, as required.
	%
	% An analogous argument shows that all triples with one vertex in $X$ and two in $Y$ are also edges, thus finishing the proof.
\end{proof}

For a (not necessarily uniform) hypergraph $\mathcal{H}$, let $\alpha_2(\mathcal{H})$ be the maximum size of a set $I \subseteq V(\mathcal{H})$ such that $|I \cap e| \leq 2$ for every $e \in E(\mathcal{H})$. Denote $g(\mathcal{H}) = \max \left( \max_{e \in E(\mathcal{H})}|e|, \alpha_2(\mathcal{H}) \right)$. Denote by $g(n)$ the minimum of $g(\mathcal{H})$ over all linear (not necessarily uniform) hypergraphs with $n$ vertices. 

\begin{lemma}\label{lem:linear hypergraph}
	$h_3(n,\{(4,2),(4,3)\}) = g(n)$.
\end{lemma}
\begin{proof}
	Let $H$ be an $n$-vertex $Q$-free $3$-graph with $h(H) = h(n,Q)$, where $Q = \{(4,2),(4,3)\}$. Let $\mathcal{H}$ be the hypergraph on $V(H)$ whose edges are the maximal cliques of $H$. Then $\mathcal{H}$ is linear by the previous lemma. Also, $\max_{e \in E(\mathcal{H})} |e| = \omega(H)$, and $\alpha_2(\mathcal{H}) = \alpha(H)$, so $h(H) = g(\mathcal{H})$.
	
	In the other direction, let $\mathcal{H}$ be an $n$-vertex linear hypergraph with $g(\mathcal{H}) = g(n)$. Let $H$ be the $3$-graph obtained by making each $e \in E(\mathcal{H})$ a clique. Then $h(H) = g(\mathcal{H})$, and it is easy to check that $H$ is $\{(4,2),(4,3)\}$-free.  
\end{proof}

From now on, our goal is to upper-bound $g(n)$. As we will shortly show, the problem can be translated to a problem about $C_4$-free bipartite graphs. We prove the following.

\begin{theorem}\label{thm:construction}
	For some positive constant $C$ and every large $m$, there is a $C_4$-free bipartite graph $G = (X, Y, E)$ with $|X| \ge \frac{1}{2} m^{3/4} \log^2 m$ and $|Y| = (1+o(1)) m$, such that the following holds:
	\begin{enumerate}
		\item $d(y) \leq 2m^{1/4}\log^2 m$ for every $y \in Y$.
		\item For every set $X' \subseteq X$ of size at least $Cm^{1/4} \log^2 m,$ there is $y \in Y$ with $|N(y) \cap X'| \ge 3.$
	\end{enumerate}
	
\end{theorem}

\begin{proof}[Proof of the upper bound in Theorem~\ref{thm:|Q|=2}(b)]
	By Lemma \ref{lem:linear hypergraph}, it is enough to show that $g(n) = \linebreak O(n^{1/3}\log^{4/3} n)$. Let $G = (X,Y,E)$ be the graph given by Theorem \ref{thm:construction}. Put $n = |X| = \Omega(m^{3/4}\log^2m)$. Let $\mathcal{H}$ be the hypergraph whose edges are the sets $N_G(y) \subseteq X$, $y \in Y$. 
	Then $\mathcal{H}$ is linear because $G$ is $C_4$-free.
	Also $\max_{e \in E(\mathcal{H})}|e| = O(m^{1/4}\log^2m) = O(n^{1/3}\log^{4/3}n)$ by Item 1 of Theorem \ref{thm:construction}. Finally, $\alpha_2(\mathcal{H}) = O(m^{1/4}\log^2m) = O(n^{1/3}\log^{4/3}n)$ by Item 2 of Theorem \ref{thm:construction}.
\end{proof}

\subsection{Proof of Theorem \ref{thm:construction}}

Let $H$ be the incidence graph of a finite projective plane with $n = (1 + o(1))m$ points and lines; that is, $H$ is a bipartite $C_4$-free graph with sides $X_0, Y$ of size $n$, and every pair of vertices in $X_0$ have exactly one common neighbour in $Y.$ Let $X$ be a random subset of $X_0$ obtained by including every vertex independently with probability 
$p = n^{-1/4} \log^2 n.$ Let $G = H[X, Y].$ Clearly, with high probability $|X| \ge \frac{3}{4}pn \geq \frac{1}{2}pm \geq \frac{1}{2} m^{3/4}\log^2 m.$
Also, we have $d_H(y) = (1+o(1))\sqrt{n}$ for every $y \in Y$, and it is easy to show, using the Chernoff bound, that w.h.p. $d(y) \leq 2\sqrt{n}p = 2n^{1/4}\log^2 n$ for every $y \in Y$. So it remains to show that w.h.p., $G$ satisfies Item 2. 
To this end, we use the container method. Let $\calI$ be the set of all subsets $I \subseteq X_0$ of size 
$Cn^{1/4} \log^2 n$ such that $|N(y) \cap I| \le 2$ for every $y \in Y.$ We want to show that with high probability $X$ contains no set in $\calI.$ We will prove the following claim.

\begin{claim} \label{claim:containers}
	There is a positive constant $C_0,$ a set $\calS \subseteq \binom{X_0}{C_0 n^{1/4} \log n}$ and a function $f \colon \calS \rightarrow \binom{X_0}{\le C_0 \sqrt{n}}$ such that for every $I \in \calI,$ there exists $S = S(I) \in \calS$ satisfying $S \subseteq I \subseteq f(S).$
\end{claim}

Let us first complete the proof given Claim~\ref{claim:containers}. Fix an arbitrary $S \in \calS.$ Note that $|X \cap f(S)|$ is distributed as $\Bin(|f(S)|,p)$. We have $\mathbb{P}[\Bin(N,p) \geq k] \leq \binom{N}{k}p^k \leq (\frac{eNp}{k})^k$. So for $k = Cn^{1/4} \log^2 n \geq \frac{C}{C_0} \cdot p|f(S)|$, we have (assuming $C \gg C_0$),
% $\E[|X \cap f(S)|] = p |f(S)| \le C_0 n^{1/4} \log n,$ by a standard Chernoff bound (assuming $C \gg C_0$), we have
% \begin{equation} \label{eq:single-S-prob}
%     \Pr\left[|X \cap f(S)| \ge Cn^{1/4} \log^2 n\right] \le \exp\left(-\frac{C}{2} n^{1/4} \log^2 n\right).
% \end{equation}
\begin{equation} \label{eq:single-S-prob}
\Pr\left[|X \cap f(S)| \ge Cn^{1/4} \log^2 n\right] \le \exp\left(-C n^{1/4} \log^2 n\right).
\end{equation}
Taking the union bound over all $S \in \calS,$ of which there are at most $\binom{n}{C_0 n^{1/4} \log n} \le \exp(2C_0 n^{1/4} \log^2 n),$ it follows that with high probability, 
$|X \cap f(S)| < Cn^{1/4} \log^{2} n$ holds for every $S \in \calS.$ Recall that for every $I \in \calI$ there is $S \in \mathcal{S}$ such that $I \subseteq f(S(I))$. Hence, for every $I \in \calI,$ we have $|I \cap X| < Cn^{1/4} \log^2 n \leq |I|,$ which implies $I \not\subseteq X,$ as required.

\begin{proof}[Proof of Claim~\ref{claim:containers}]
    We present an algorithm which, given $I,$ produces sets $S(I) \subseteq I$ and $f(S) \supseteq I.$ The algorithm maintains sets $A^t,S^t$. Initially, we set $A^0 = X_0, S^0 = \emptyset$. The algorithm runs for $q = C_0 n^{1/4} \log n$ steps $t=0, \dots, q-1$ and in step $t$, obtains an index $i^t,$ to be defined later, and new sets $A^{t+1}, S^{t+1}.$ Recall that for any $I \in \calI,$ we have $|I| = Cn^{1/4} \log^2n > q.$ Throughout the algorithm we will have $|S^t| = t, S^t \subseteq I \subseteq S^t \cup A^t$ and $S^t \cap A^t = \emptyset.$ Now, suppose we are at step $t$. We define a graph $F^t$ with $V(F^t) = A^t$ and where $aa' \in E(F^t)$ if and only if there exist $s \in S^t$ and $y \in Y$ such that $a, a', s \in N(y).$ Note that $F^t$ only depends on $A^t, S^t,$ but not on $I.$ Let $a^t_1, a^t_2, \dots, a^t_{|A_t|}$ be an ordering of $A_t$ such that for all $i,$ $a^t_i$ is a vertex of maximum degree in $F^t[\{a^t_i, a^t_{i+1}, \dots, a^t_{|A_t|}\}],$ with ties broken according to some fixed ordering of $X_0.$ Let $i^t$ be the minimum index $i$ such that $a^t_i \in I.$ We let $S^{t+1} = S^t \cup \{a^t_{i^t}\}$ and $A^{t+1} = A^t \setminus ( \{a^t_1, \dots, a^t_{i^t}\} \cup N_{F^t}(a_{i^t})).$ Note that $i^t$ is well-defined since we have $|S^t| < q < |I|$ and $I \subseteq S^t \cup A^t$ (which we will soon prove). After $q$ steps, we let $S(I) = S_q$ and $f(S_q) = S_q \cup A_q.$ We denote $\calS = \{ S(I) \, \vert \, I \in \calI \}.$
	
	Clearly, we have $S^t \subseteq I, S^t \cap A^t = \emptyset$ for any $t \in \{0, \dots q\}$ and $S^{t+1} \cup A^{t+1} \subseteq S^t \cup A^t$ for any $t \in \{0, \dots, q-1\}.$ Let us also verify that $I \subseteq S^t \cup A^t$ throughout, which clearly implies that $I \subseteq f(S(I))$. Indeed, suppose that $I \subseteq S^t \cup A^t$ at some step of the algorithm, let $a^t_1, \dots, a^t_{|A_t|}$ be the ordering of $A^t$ as described in the algorithm, and let $i = i^t$ be the index chosen in the algorithm, i.e. such that $a^t_i \in I$ and $a^t_1,\dots,a^t_{i-1} \notin I$. Consider a neighbour $v$ of $a^t_i$ in $F^t.$ By definition, there exist $s \in S^t \subseteq I$ and $y \in Y$ such that $s, a^t_i, v \in N(y).$ Then, since $I \in \calI$ and $s, a^t_i \in I,$ it follows that $v \not\in I.$ Hence, $I \subseteq A^{t+1} \cup S^{t+1}$.
	
	Let us now prove that $f(S)$ is indeed uniquely determined by $S.$ In the following, we will denote by $S^t(I), A^t(I)$ the relevant $S^t, A^t$ when the input of the algorithm is $I$, and similarly denote by $i^t(I)$ the relevant index $i^t$.
	Fix $I,I' \in \mathcal{I}$ such that $S(I) = S(I')$. We show that $S^t(I) = S^t(I')$ and $A^t(I) = A^t(I')$ for all $t \in \{0, \dots, q\}.$ This clearly holds for $t = 0$. Suppose that this holds for some $t$, and let us prove this for $t+1$. 
	% We show that throughout the algorithm $A_t(I)$ is uniquely determined by $S_t(I).$ We proceed by induction on $t.$ The statement clearly holds for $t = 0.$ 
	% Suppose now that there are sets $I, I' \in \calI$ such that $S_t(I) = S_t(I')$, $A_t(I) = A_t(I')$, and $S_{t+1}(I) = S_{t+1}(I')$.
    Denote $S^t = S^t(I) = S^t(I'), A^t = A^t(I) = A^t(I')$ and $F^t = F^t(I) = F^t(I'),$ where the last equality holds since $F^t(J)$ is uniquely determined by $S^t(J)$ and $A^t(J).$ Let $a^t_1, \dots, a^t_{|A^t|}$ be the ordering of $A^t = V(F^t)$ as above. Denote $i = i^t(I)$ and $i' = i^t(I').$ If $i = i',$ then it follows that $S^{t+1}(I) = S^{t+1}(I')$ and from the definition of the algorithm, also $A^{t+1}(I) = A^{t+1}(I'),$ as required. So let us assume without loss of generality that $i < i'.$ Then, $a^t_i \in S^{t+1}(I) \subseteq S(I).$ On the other hand, by definition of $i',$ we have $a^t_i \not\in I'$ which, using that $S(I') \subseteq I',$ implies $a^t_i \not\in S(I').$ Hence, $S(I) \neq S(I'),$ contradicting our assumption.
 
    %By definition, $s_I$ (resp. $s_{I'}$) is the element of $S(I) \setminus S_t(I)$ (resp. $S(I') \setminus S_t(I')$) which comes earliest in the ordering $a_1, \dots, a_{|A_t|}$. But we assumed that $S(I) = S(I')$ and $S_t(I) = S_t(I')$, so $s_I = s_{I'}$ and hence $S_{t+1}(I) = S_{t+1}(I')$.
	%It now follows directly that also $A_{t+1}(I) = A_{t+1}(I')$, as required.
	
	Finally, we need to show that $|f(S)| \le C_0 \sqrt{n}$ for every $S \in \calS$. We will prove the following claim.
	
	\begin{claim} \label{claim:decrease}
		Suppose that $t \ge 2n^{1/4}$ and $|A^t| \ge 10 \sqrt{n}.$ Then $|A^{t+1}| \le (1-n^{-1/4}) |A^t|$.
	\end{claim}
	
	Let us finish the proof given Claim~\ref{claim:decrease}. Fix any $I \in \calI$ and $S = S(I)$, and suppose for the sake of contradiction that $|f(S)| = |S \cup A^q(I)| \geq 11 \sqrt{n}.$  As 
	$|S| = q \ll \sqrt{n}$, we must have $|A^q(I)| \geq 10\sqrt{n}$. 
	Then, by Claim~\ref{claim:decrease}, for any $t \in [2n^{1/4}, q-1],$ we have $|A^{t+1}| \le (1-n^{-1/4}) |A^t|,$ which implies
	\[ |A^q| \le n \cdot \left( 1 - n^{-1/4} \right)^{q - 2n^{1/4}} < n \cdot e^{-n^{-1/4} \cdot (q / 2)} \leq n \cdot e^{-\log n} = 1, \]
	a contradiction.
	
	\begin{proof}[Proof of Claim~\ref{claim:decrease}]
		Let $S^t, A^t, F^t, a^t_1, \dots, a^t_{|A^t|}$ and $i^t$ be as given in the algorithm. For $1 \le j \le |A^t|,$ denote $F_j = F^t[\{a^t_{j}, \dots, a^t_{|A^t|}\}]$. It is enough to prove that $\Delta(F_j) \ge |V(F_j)|/n^{1/4}$ for every $j \le |A^t| / 2.$ Indeed, then if $i^t \le |A^t| / 2,$ we obtain $|A^{t+1}| \le |A^t| - i^t - \Delta(F_{i^t}) \le
		|A^t| - i^t - |V(F_{i_t})|/n^{1/4} = 
		|A^t| - i^t - (|A^t| - i^t + 1)/n^{1/4} \le
		(1 - n^{-1/4}) |A^t|$, and if $i^t \ge |A^t| / 2,$ then $|A^{t+1}| \le |A^t| / 2.$
		
		Consider a fixed $1 \leq j \leq |A^t|/2.$ Denote $A' = \{a_j, \dots, a_{|A^t|}\} = V(F_j).$ 
		We need to show that $\Delta(F_j) \geq |A'|/n^{1/4}$.
		Fix any $s \in S^t.$ Then, for every $y \in N_H(s)$ and distinct $a, a' \in A' \cap N_H(y),$ we have $aa' \in E(F_j)$. Note that the sets $(N_H(y) \setminus \{s\})_{y \in N_H(S^t)}$ partition $X_0 \setminus \{s\}$, since every two vertices in $X_0$ have exactly one common neighbour in $H$. The number of pairs $(y, \{a, a'\})$ with $a, a' \in A'$ and $a, a', s \in N_H(y)$ is
		\[ \sum_{y \in N_H(s)} \binom{ |A' \cap N_H(y)|}{2} \ge |N_H(s)| \cdot \binom{|A'| / |N_H(s)|}{2} \ge \frac{|A'|^2}{4 \sqrt{n}}, \]
		where we used Jensen's inequality for the convex function $\binom{x}{2}$, the fact that $N_H(s) = (1 + o(1)) \sqrt{n}$, 
		and the assumption that $|A'| \geq |A^t|/2 \geq 5\sqrt{n}$.
		Hence, every $s \in S^t$ contributes at least $\frac{|A'|^2}{4 \sqrt{n}}$ edges to $F_j$. Finally, we prove that for every $aa' \in E(F_j),$ there are unique $s \in S^t, y \in Y$ such that $s, a, a' \in N_H(y)$. Indeed, recall that every pair of vertices in $X_0$ have a unique common neighbour in $Y.$ Hence, given $a, a',$ the vertex $y \in Y$ is uniquely determined. But then, the vertex $s \in S^t \cap N_H(y)$ is also uniquely determined. Indeed, suppose there are two distinct $s, s' \in S^t \cap N_H(y).$ Without loss of generality, there is an index $t_0$ such that $\{s\} = S^{t_0} \setminus S^{t_0-1}$ and $s' \in S^{t_0-1}.$ Then, by definition, $sa, \in E(F^{t_0-1}),$ so $a \not\in S^{t_0} \cup A^{t_0} \supseteq S^t,$ a contradiction.
		
		Therefore, we have $e(F_j) \ge |S^t| \cdot \frac{|A'|^2}{4 \sqrt{n}} \ge \frac{|A'|^2}{2 n^{1/4}},$ which implies that $\Delta(F_j) \ge |A'| / n^{1/4}$ as required.                        
	\end{proof}
	This concludes the proof of Claim~\ref{claim:containers} and hence the theorem.
\end{proof}

\section{$(4,1),(4,3)$: Proof of Theorem~\ref{thm:|Q|=2}(d)}\label{sec:13}
We need the following lemma. 
\begin{lemma}\label{lem:random graph}
	There is a constant $C > 0$ such that for every $n$, there is an $n$-vertex graph in which every set of size $C\log n$ contains a triangle and an independent set of size $3$. 
\end{lemma}
\begin{proof}
    Take $G \sim G(n,1/2)$. Fix any $U \subseteq V(G)$, $|U| = k := C\log n$. It is well-known that there is a partial Steiner system on $U$ with $m = (\frac{1}{6}-o(1))k^2 \geq k^2/7$ triples, $T_1,\dots,T_m$. The probability that no $T_i$ is a triangle in $G$ is $(7/8)^{m} \leq (7/8)^{k^2/7} = (7/8)^{C^2\log^2 n}$. Taking the union bound over all $\binom{n}{C\log n} \leq e^{C\log^2 n}$ choices for $U$, and assuming that $C$ is large enough, we get that with high probability, every set of size $C\log n$ contains a triangle. By the same argument, w.h.p. every such set contains an independent set of size $3$.  
\end{proof}
\begin{proof}[Proof of Theorem~\ref{thm:|Q|=2}(d)]
For the lower bound, let $H$ be an $n$-vertex $\{(4,1),(4,3)\}$-free  $3$-graph. Pick a vertex $v$ in $H$ and consider its link graph $L(v)$. Since $R_2(t,t) < 4^{t-1}$ (see Erd\H os and Szekeres~\cite{ES}),  we see that $L(v)$ has a clique or coclique $K$ of size at least $\frac{1}{2} \log n$. In the first case,  $K$ is a clique in $H$, else we find a $(4,3)$-subgraph in $H$; and in the second case,  $K$ is a coclique in $H$, else we find a $(4,1)$-subgraph in $H$.

For the upper bound, let $G$ be the graph from Lemma~\ref{lem:random graph}. 
Let $H$ be the $3$-graph  on vertex set $V(G)$ whose edge set consists of all triples of vertices $x,y,z$ which induce an odd number of edges in $G$. Lemma~\ref{lem:random graph} guarantees that every set of $C\log n$ vertices contains both an edge and a non-edge of $H$. Hence, $h(H) \leq C\log n$.   
Let us show that $H$ is $Q$-free, $Q = \{(4,1),(4,3)\}$. 
Fix any $X \subseteq V(G) = V(H)$, $|X|=4$. 
For each $A \subseteq X$, $|A| = 3$, we have $A \in E(H)$ if and only if $e_G(A)$ is odd, where $e_G(A)$ is the number of edges spanned by $A$ in $G$. Note that each edge of $G[X]$ is contained in exactly two sets $A \subseteq X$, $|A| = 3$. Hence,
$
\sum_{A \subseteq X, |A|=3}e_G(A) = 2e_G(X).
$
The right-hand side is even, so there is an even number of $A$ with $e_G(A)$ odd. It follows that every four vertices in $H$ induce an even number of edges. So $H$ is $Q$-free. 
\end{proof}

\section{Forbidden sets of size $3$: Proof of Theorem~\ref{|Q|=3}}\label{sec:|Q|=3}
%%%%%%%%%%%%%%%%%%%%%%%%%%%%%%%%%%%%%%%%%%%%%%%%%%%%%%%%%%%
%It is convenient to use the notation $\cForb(n, Q)$ for the family of $n$-vertex $Q$-free $3$-graphs.
 We will need the following structural characterization of $Q$-free $3$-graphs for 
$Q= \{ (4,1), (4,3), (4,4) \}$.
\begin{theorem}[Frankl and F\"uredi~\cite{FF}]\label{lem:ff_characterisation}
	Let $H$ be an $\{ (4,1), (4,3), (4,4) \})$-free $3$-graph. Then $H$ is isomorphic to one of the following $3$-graphs: 
	\begin{enumerate}
		\item A blow-up of the $6$ vertex $3$-graph $H'$ with vertex set $V(H') = [6]$ and edge set $E(H') = \{123, 124, 345, 346, 561, 562, 135, 146, 236, 245\}$. Here for the blow-up we replace every vertex of $H'$ by an independent set, and whenever we have $3$ vertices from three distinct of those sets, they induce an edge if and only if the corresponding vertices in $H'$ do.
		\item The $3$-graph whose vertices are the points of a regular $n$-gon where $3$ vertices span an edge if and only if the corresponding points span a triangle whose interior contains the center of the $n$-gon.
	\end{enumerate} 
\end{theorem}

\noindent
{\bf Proof of Theorem~\ref{|Q|=3}.}

{\it Case $Q = \{ (4,1), (4,3), (4,4) \}$.}\\
 We are to prove that
$$ h(n, \{(4,0), (4,1), (4,3)\})= h(n, Q) =\begin{cases}
	\frac{n}{2}  &\text {if $n \equiv 0$ (mod 6)} \\
	\lceil \frac{n+1}{2}\rceil & \text {if 	$n \not\equiv 0$ (mod 6)}.
\end{cases}	 $$
First, let us prove that 
the second 3-graph $H$ in Theorem~\ref{lem:ff_characterisation} has independence number exactly $\lceil{(n+1)/2)}\rceil$. Assume  the vertex set is $[n]$ and the vertices are labeled by consecutive integers in clockwise orientation. The lower bound is by taking $\lceil{(n+1)/2)}\rceil$ consecutive vertices on the $n$-gon and noting that no three of them contain the center in their interior.   For the upper bound, let us see how many  vertices can lie in an independent set containing $1$. When $n$ is odd,  the triangle formed by  $\{1, i, (n-1)/2+i\}$  contains the center and hence is  an edge. Therefore we may pair the elements of $[n]\setminus\{1\}$ as $(2, (n+3)/2), (3, (n+5)/2), \ldots, ((n+1)/2, n)$ and note that each pair can have at most one vertex in an independent set containing 1. Hence the maximum size of an independent set containing 1 is at most $(n+1)/2$ and by vertex transitivity of $H$, the independence number of $H$ is at most $(n+1)/2$. For $n$ even we consider the $n/2-1$ pairs $(2, n/2+1), (3, n/2+2), \ldots, (n/2, n-1)$ and add the vertex $n$ to get an upper bound $n/2+1=\lceil{(n+1)/2)}\rceil$.

Next we observe that the $6$-vertex 3-graph $H'$ in Theorem~\ref{lem:ff_characterisation} has independence number exactly $3$ (we omit the short case analysis needed for the proof). Hence if we blow-up each vertex of $H'$ into sets of the same size, then we obtain $n$-vertex $3$-graphs with independence number exactly $n/2$ whenever  $n \equiv 0$ (mod 6). This concludes the proof of the upper bound.

For the lower bound, let  $H$ be $Q$-free. Then by Theorem \ref{lem:ff_characterisation}, $H$ is isomorphic to one of the two graphs described in  Theorem~\ref{lem:ff_characterisation}. If $H$ is isomorphic to the second graph, then we have already shown that its independence number is at least $(n+1)/2$, so assume that 
$H$ is isomorphic to the blow-up of the $6$-vertex $10$-edge $3$-graph $H'$.   There are $10$ non-edges in $H'$. Let $V_1, \ldots, V_6$ be the blown up vertex sets.   Since every vertex $i \in [6]$ in $H'$ is contained in exactly $5$ non-edges, we obtain 
$$ 5n = 5\sum\limits_{i \in [6]} |V_i| = \sum\limits_{j_1j_2j_3 \not\in E(H)} |V_{j_1}| + |V_{j_2}| + |V_{j_3}|.$$
By the pigeonhole principle, there is a non-edge $i_1i_2i_3$, such that $ |V_{i_1}| + |V_{i_2}| + |V_{i_3}|\geq n/2$.  Our bound follows by observing that  for any  non-edge $i_1i_2i_3$ in the original $3$-graph $H'$ the set $V_{i_1} \cup V_{i_2} \cup V_{i_3}$ is an independent set. This gives an independent set of size at least $n/2$, and if 	$n \not\equiv 0$ (mod 6), then equality cannot hold throughout (a short case analysis, which we omit,  is needed to prove this)  and we obtain an independent set of size strictly greater than $n/2$ as required.\\

{\it Case $Q = \{ (4,0), (4,2), (4,3) \}$.}\\
We now prove $h(n, \{ (4,0), (4,2), (4,3) \} )  = n-1,$ for $n\geq 4$.
Let $H$ be a $3$-graph  that is a  clique on $n-1$ vertices and a single isolated vertex, then $H$ is $Q$-free, giving us the upper bound. 

 For the lower bound, let $H$ be a $Q$-free $3$-graph on $n$ vertices, $n\geq 4$. 
 Assume that $H$ is not a clique. We shall show that  $H$ is a clique and a single isolated vertex.  Consider a maximal clique $S$ in $H$. Since  $|S|<n$, there is  a vertex $v\in V(H)\setminus S$. From the  maximality of $S$,  $L_S(v) $ is not a clique. If $L_S(v)$ contains an edge, then we have that for some vertices $x, y, y'$, 
 $xy\in E(L_S(v))$ and $xy'\not\in E(L_S(v))$. But then $\{v, x, y, y'\}$ induces a $(4,2)$ or a $(4,3)$-graph, a contradiction. Thus, $L_S(v)$ is an empty graph, i.e., there is no edge in $H$ containing $v$ and two vertices of $S$. 
Now assume there exists a second vertex $v' \in V(H)\setminus(S \cup \{v\})$. Then by the same argument as above, $v'$ is also not contained in any edge with two vertices from $S$. 
Consider triples $vv'x$, $x\in S$. Since $|S|\ge 3$, by the pigeonhole principle there are  two vertices $x,x'\in S$ such that either
$vv'x, vv'x'\in E(H)$ or $vv'x, vv'x'\not\in E(H)$. Then $\{v,v', x, x'\}$ induces $2$ or $0$ edges respectively, a contradiction. Thus, $|S|=n-1$ and $v$ is an isolated vertex.
 \qed

\section{Concluding Remarks}

Fix integers $m>r$. Say that a set $Q$ of order size pairs $\{(m, f_1), \ldots, (m,f_t)\}$  is Erd\H os-Hajnal (EH) if there exists $\epsilon=\epsilon_Q$ such that $h_r(n, Q)>n^{\epsilon}$.  As $|Q|$ grows, the collection of $Q$-free $r$-graphs is more restrictive, and hence $h_r(n, Q)$ grows (assuming that large $Q$-free $r$-graphs are not forbidden to exist by Ramsey's theorem). 
The case when $h_r(n, Q) = \Omega(n)$ was treated by the first author and Balogh~\cite{AB} when $r=2$.
A natural question then is to ask what is the smallest $t$ such that every $Q$ of size $t$ is EH. Call this minimum value $EH_r(m)$.  Our results for $r=3$ show that for $m=4$, all $Q$ of size 3 are EH, but there are $Q$ of size 2 which are not EH. Consequently, $EH_3(4) = 3$.  

In order to further study $EH_r(m)$, we need another definition. Given integers $m\ge r\ge 3$, let $g_r(m)$ be the number of edges in an $r$-graph on $m$ vertices obtained by first taking a partition of the $m$ vertices into almost equal parts, then taking all edges that intersect each part, and then recursing this construction within each part. For example, $g_3(7)= 13$ since we start with a complete 3-partite 3-graph with part sizes $2,2,3$ and then add one edge within the part of size 3. It is known (see, e.g.~\cite{MR}) that as $r$ grows we have
$$g_r(m) = (1+o(1))\frac{r!}{r^r-r} {m \choose r}.$$
Note that $\frac{r!}{r^r-r}$ approaches 0 as $r$ grows. The second author and Razborov~\cite{MR} proved that for all fixed $m>r> 3$, there are $n$-vertex $r$-graphs
which are $Q$-free, $Q=\{(m, i): g_r(m)<i\le {m \choose r}\}$, with $h(G)= O(\log n)$. In other words, there exists $Q$ of size ${m \choose r} - g_r(m)$ which is not EH.
This proves that $EH_r(m)  \ge {m \choose r} - g_r(m)+1$. 

 Erd\H os and Hajnal~\cite{EH1} proved that for all $m > r \ge 3$, the set $Q=\{(m, i): g_r(m) \le i \le {m \choose r}\}$ is EH. In other words, they proved that every $n$-vertex $r$-graph in which every set of $m$ vertices spans less then $g_r(m)$ edges has an independent set of size at least $n^{\epsilon}$, where $\epsilon$ depends only on $r$ and $m$.  This is a particular set $Q$ of size ${m \choose r} - g_r(m) +1$
that is EH and we speculate that every other set $Q$ of this size is also EH.

\begin{problem} Prove or disprove that for all $m>r>2$,
	$$EH_r(m) = {m \choose r} - g_r(m) +1.$$
	\end{problem}
We end by noting that  $EH_3(4) = 3 = {4 \choose 3} - g_3(4) +1$. 

\bigskip \medskip 

{\bf Acknowledgments:} The second and third authors thank Benny Sudakov for useful discussions.

\end{document}